\documentclass[12pt]{amsart}
\usepackage{geometry} % see geometry.pdf on how to lay out the page. There's lots.
\newtheorem{theorem}{Theorem}
\newtheorem{corollary}{Corollary}
\title{Hankel weighing matrices}
\author{Goldwyn Millar\\ School of Mathematics and Statistics\\ Carleton University\\ Ottawa, ON\\} 
\date{December 8, 2013} % delete this line to display the current date

\begin{document}
\begin{abstract} We prove a classification theorem for Hankel weighing matrices.
\end{abstract}
\maketitle
Let $n,k \in \mathbb{N}.$ A \textit{weighing matrix of order $n$ and weight $k$} (a $W(n,k)$) is a $n \times n$ matrix $A$ with $0,\pm 1$ entries and such that $AA^T = kI.$ A Hankel matrix is a matrix whose skew diagonals are constant. One can obtain a weighing matrix whose diagonals are constant from a weighing matrix whose skew diagonals are constant (and vice versa) by inverting the order of the rows of the matrix (since orthogonality of rows is invariant under row permutation). It will be convenient for us to define a Hankel weighing matrix as a weighing matrix whose diagonals are constant. So, to be precise, we stipulate that a $n \times n$ weighing matrix $A = [a_{ij}]$ is a \textit{Hankel weighing matrix} (a $HW(n,k)$) if for each entry $a_{ij}$ of $A$ such that $2 \leq i,j \leq n,$ $a_{ij} = a_{i - 1,j - 1}.$ Recently, it has been suggested that Hankel weighing matrices might have applications in the theory of quantum error-correcting codes [see \cite{F1} and \cite{F2}]. \\
 A \textit{circulant matrix} is a matrix such that each of its rows, after the first, can be obtained from the row above it by a right cyclic shift. Much is known about circulant weighing matrices ($CW(n,k)$'s) [see, for instance, the surveys \cite{A1} and \cite{M}]. The authors of \cite{F1} remark that although circulant weighing matrices are Hankel weighing matrices, there are Hankel weighing matrices that are not circulant weighing matrices, so that the problem of classifying Hankel weighing matrices does not reduce entirely to the problem of classifying circulant weighing matrices.\\
A \textit{negacyclic matrix} is a matrix that can be written as a polynomial in the matrix $Y$ given below:
   \[Y = \left(\begin{matrix} % or pmatrix or bmatrix or Bmatrix or ...
      0  &  1  &  0  &  \cdot  &  \cdot  &  \cdot  &  0 \\
      0  &  0  &  1  &  0  &  \cdot  &  \cdot  \\
         &   &   &\cdot\\
         &  &  &\cdot\\
         &  &  &\cdot & & & 1\\
      -1 &  0 &  0  &  \cdot  &  \cdot  &  \cdot   &  0\\
   \end{matrix}\right).\]
Negacyclic weighing matrices ($NW(n,k)$'s) are not as well studied as $CW(n,k)$'s, but several papers have been written about them over the years [see \cite{D}, \cite{B}, and \cite{A2}]. Clearly, $NW(n,k)$'s are $HW(n,k)$'s.\\
It turns out that with respect to $HW(n,k)$'s, the obvious sufficient conditions are also necessary. 
\begin{theorem} Every $HW(n,k)$ is either a $CW(n,k)$ or a $NW(n,k).$
\end{theorem}
\begin{proof} Let $A$ be a $HW(n,k)$. Since there are $k$ nonzero entries in each row of A, we have that for each  $2 \leq i \leq n,$ $\displaystyle{\sum_j |a_{ij}| = \sum_j |a_{i - 1,j}|}$. So, $|a_{i,1}| = |a_{i-1,n}|$, i.e. $a_{i - 1,n} = \epsilon_i a_{i1},$ where $\epsilon_i = \pm 1.$ Since the dot product of any pair of rows is equal to zero, we have that for any $2 \leq i,t \leq n$,  $\displaystyle{\sum_j a_{i,j}a_{t,j} = \sum_j a_{i-1,j}a_{t-1,j}}.$ Thus, since $A$ is a Hankel matrix, $a_{i,1}a_{t,1} = a_{i-1,n}a_{t-1,n} = \epsilon_i a_{i,1} \epsilon_t a_{t,1}.$ Therefore, when $a_{i,1},a_{t,1} \neq 0,$ $\epsilon_i \epsilon_t = 1.$ Since $i$ and $t$ were chosen arbitrarily, it follows that either $a_{i,1} = a_{i-1,n}$ for each $2\leq i \leq n$, or $a_{i,1} = -a_{i-1,n}$ for each $2 \leq i \leq n.$ Hence, $A$ is either a $CW(n,k)$ or a $NW(n,k).$ 
\end{proof}
The above theorem and the applications of Hankel weighing matrices alluded to earlier provide incentive for further research into the existence/non-existence of negacyclic weighing matrices. For now, we are content to deduce an easy corollary.\\
A conference matrix is a weighing matrix of order $n$ and weight $n-1.$ It is well known that for $n>2,$ there exist no circulant conference matrices of order $n$ \cite{Mul}. Delsarte, Goethals, and Seidal  showed that if $q$ is an odd prime power, then there exists a negacyclic conference matrix of weight $q$ \cite{D}. Furthermore, they conjectured that if $n-1$ is not an odd prime power, then there exists no conference matrix of weight $n-1.$ They were able to prove their conjecture for $n \leq 227.$ It has since been shown that their conjecture is correct for $n \leq 532$ \cite{A3}. 
\begin{corollary} For $2 < n \leq 532,$ there exists a Hankel conference matrix of order $n$ if and only if $n-1$ is an odd prime power.
\end{corollary}
This note was written while the author was a MSc student at the University of Manitoba under the helpful supervision of Dr. R. Craigen. The author became aware of Hankel weighing matrices when Dr. Craigen  informed him of the contents of an e-mail he had received on the subject \cite{S}.

\end{document}